\documentclass[11pt,reqno]{amsart}

\usepackage{xspace}
\usepackage{amsmath,amsthm,amsfonts,amssymb,latexsym,mathrsfs,color}
\usepackage{hyperref}

\newtheorem{theorem}{Theorem}
\newtheorem{cor}[theorem]{Corollary}
\newtheorem{prop}[theorem]{Proposition}

\newtheorem{ex}[theorem]{Example}

\newtheorem{definition}[theorem]{Definition}

\newcommand{\ap}{{\rm ap\,}}
\newcommand{\asc}{{\rm asc\,}}

\newcommand{\cyc}{{\rm cyc\,}}

\newcommand{\exc}{{\rm exc\,}}

\newcommand{\cdes}{{\rm cdes\,}}

\newcommand{\nega}{{\rm neg\,}}
\newcommand{\pos}{{\rm pos\,}}
\newcommand{\st}{{\rm star\,}}

\newcommand{\oo}{{\rm oo\,}}
\newcommand{\so}{{\rm so\,}}

\newcommand{\oee}{{\rm oe\,}}
\newcommand{\ee}{{\rm ee\,}}
\newcommand{\eo}{{\rm eo\,}}
\newcommand{\mq}{\mathcal{Q}}
\newcommand{\mqn}{\mathcal{Q}_n}

\newcommand{\msn}{\mathfrak{S}_n}

\newcommand{\msnn}{\mathfrak{S}_{n+1}}
\newcommand{\z}{ \mathbb{Z}}

\newcommand{\my}{{\rm MY}}

\newcommand{\mmn}{\mathcal{M}_{2n}}
\newcommand{\mmnn}{\mathcal{M}_{2n+2}}

\newcommand{\m}{{\rm M}}
\newcommand{\mm}{\mathcal{M}}

\newcommand{\myn}{\mathcal{Y}_{n}}
\newcommand{\mynn}{\mathcal{Y}_{n-1}}

\newcommand{\OO}{{\rm OO}}
\newcommand{\OEE}{{\rm OE}}
\newcommand{\EE}{{\rm EE}}
\newcommand{\EO}{{\rm EO}}
\newcommand{\BBA}{{\rm BB algorithm}}
\newcommand{\LSPA}{{\rm LSP algorithm}}
\newcommand{\LPA}{{\rm LP algorithm}}

\newcommand{\arxiv}[1]{\href{http://arxiv.org/abs/#1}{\texttt{arXiv:#1}}}

\title{Stirling permutations, cycle structures of permutations and perfect matchings}
\author[S.-M.~Ma]{Shi-Mei Ma}
\address{School of Mathematics and Statistics,
        Northeastern University at Qinhuangdao,
         Hebei 066004, P.R. China}
\email{shimeimapapers@gmail.com (S.-M. Ma)}
\author[Y.-N. Yeh]{Yeong-Nan Yeh}
\address{Institute of Mathematics,
        Academia Sinica, Taipei, Taiwan}
\email{mayeh@math.sinica.edu.tw (Y.-N. Yeh)}

\begin{document}

\maketitle
\begin{abstract}
In this paper we provide a unified combinatorial approach to establish a
connection between Stirling permutations, cycle structures of permutations and perfect matchings.
The main tool of our investigations is $\my$-sequences.
In particular, we discover that the Eulerian polynomials have a simple combinatorial interpretation in terms of some statistics on $\my$-sequences.
\end{abstract}
\date{\today}
\section{Introduction}
Stirling permutations were defined by Gessel and Stanley~\cite{Gessel78}.
Let $j^2:=\{j,j\}$ for $j\geq 1$.
A {\it Stirling permutation} of order $n$ is a permutation of the multiset $\{1^2,2^2,\ldots,n^2\}$ such that
every element between the two occurrences of $i$ are greater than $i$ for each $i\in [n]$,
where $[n]=\{1,2,\ldots,n\}$.
Denote by $\mqn$ the set of Stirling permutations of order $n$.
For $\sigma=\sigma_1\sigma_2\cdots\sigma_{2n}\in\mqn$, an occurrence of
an {\it ascent} (resp. {\it~a plateau}) is an index $i$
such that $\sigma_i<\sigma_{i+1}$ (resp. $\sigma_i=\sigma_{i+1}$),
Recently, there is a large literature devoted to Stirling permutations and their generalizations.
The reader is referred to~\cite{Bona08,Haglund12,Janson11,Ma14,Remmel14} for recent
progress on the study of statistics on Stirling permutations.

In this paper, we always assume that Stirling permutations are prepended by 0.
That is, we identify an $n$-Striling permutation
$\sigma_1\sigma_2\cdots\sigma_{2n}$ with the word $\sigma_0\sigma_1\sigma_2\cdots\sigma_{2n}$,
where $\sigma_0=0$.
Let $\sigma=\sigma_1\sigma_2\cdots\sigma_{2n}\in \mqn $.
We say that an index $i\in [2n-1]$ is an {\it ascent plateau} if $\sigma_{i-1}<\sigma_i=\sigma_{i+1}$ (see~\cite{Ma14}).
Let $\ap(\sigma)$ be the number of the ascent plateaus of $\sigma$. For example, $\ap(\textbf{2}211\textbf{3}3)=2$.

We define
\begin{equation}\label{Nnx-Stirling}
N_n(x)=\sum_{\sigma \in \mqn}x^{\ap(\sigma)}.
\end{equation}
Analyzing the placement of $2$ copies of $(n+1)$, it is easy to deduce that the polynomials $N_n(x)$ satisfy the recurrence relation
$$N_{n+1}(x)=(2n+1)xN_n(x)+2x(1-x)N'_n(x)$$
with the initial value $N_0(x)=1$.
The exponential generating function for $N_n(x)$ is given as follows (see~\cite[Section~5]{Ma13}):
\begin{equation}\label{EGF-Nnx}
N(x,z)=\sum_{n\geq 0}N_n(x)\frac{z^n}{n!}=\sqrt{\frac{1-x}{1-xe^{2z(1-x)}}}.
\end{equation}
The first few of $N_n(x)$ are
$$N_1(x)=x,
N_2(x)=2x+x^2,
N_3(x)=4x+10x^2+x^3,
N_4(x)=8x+60x^2+36x^3+x^4.$$

Let $\msn$ denote the permutation group on the set $[n]$
and $\pi=\pi(1)\pi(2)\cdots \pi(n)\in\msn$. An {\it excedance} in $\pi$ is an
index $i$ such that $\pi(i)>i$. Let $\exc(\pi)$ denote the number of
excedances in $\pi$. The classical
Eulerian polynomials $A_n(x)$ are defined by
$$A_0(x)=1,\quad A_n(x)=\sum_{\pi\in\msn}x^{\exc(\pi)}\quad\textrm{for $n\ge 1$},$$
and have been extensively investigated (see~\cite{Dilks09,FS70,Gessel93} for instance).
In~\cite{FS70},
Foata and Sch\"utzenberger introduced a $q$-analog of the
Eulerian polynomials defined by
\begin{equation*}\label{anxq-def}
A_n(x;q)=\sum_{\pi\in\msn}x^{\exc(\pi)}q^{\cyc(\pi)}.
\end{equation*}
where $\cyc(\pi)$ is the number of cycles in $\pi$.
Brenti~\cite{Bre94,Bre00} further studied $q$-Eulerian polynomials
and established the link with $q$-symmetric functions arising from
plethysm. In particular, Brenti~\cite[Proposition 7.3]{Bre00} obtained the exponential generating function for $A_n(x;q)$:
\begin{equation*}\label{anxq-rr}
1+\sum_{n\geq 1}A_n(x;q)\frac{z^n}{n!}=\left(\frac{1-x}{e^{z(x-1)}-x}\right)^q.
\end{equation*}

For $k\geq 1$, the {\it $1/k$-Eulerian polynomials} $A_n^{(k)}(x)$ are defined by
\begin{equation}\label{Ankx-def01}
\sum_{n\geq 0}A_n^{(k)}(x)\frac{z^n}{n!}=\left(\frac{1-x}{e^{kz(x-1)}-x} \right)^{\frac{1}{k}}.
\end{equation}
Let $e=(e_1,e_2,\ldots,e_n)\in\z^n$. Let
$I_{n,k}=\left\{ e|0\leq e_i\leq (i-1)k\right\}$, which known as the set of $n$-dimensional {\it $k$-inversion sequences} (see~\cite{Savage1201}).
The number of {\it ascents} of $e$ is defined by
$$\asc(e)=\#\left\{i:1\leq i\leq n-1\big{|}\frac{e_i}{(i-1)k+1}<\frac{e_{i+1}}{ik+1}\right\}.$$
Recently, Savage and Viswanathan~\cite{Savage1202} discovered that
$$A_n^{(k)}(x)=\sum_{e\in I_{n,k}}x^{\asc(e)}=k^nA_n(x;1/k).$$

From~\eqref{EGF-Nnx} and~\eqref{Ankx-def01}, we get
\begin{equation}\label{An2x-Nnx}
A_n^{(2)}(x)=x^{n}N_n\left(\frac{1}{x}\right).
\end{equation}
Hence
\begin{equation}\label{Nnx-exc-cyc}
N_n(x)=\sum_{\pi\in\msn}x^{n-\exc(\pi)}2^{n-\cyc(\pi)}.
\end{equation}
Let $\widehat{N}_n(x)=x^{n}N_n\left(\frac{1}{x}\right)$.
It follows from~\eqref{Ankx-def01} and~\eqref{An2x-Nnx} that
\begin{equation*}\label{Anx-Nnx01}
2^nA_n(x)=\sum_{k=0}^n\binom{n}{k}\widehat{N}_k(x)\widehat{N}_{n-k}(x).
\end{equation*}

A {\it perfect matching} of $[2n]$ is a set partition of $[2n]$ with blocks (disjoint nonempty subsets) of size exactly 2.
Let $\mmn$ be the set of matchings of $[2n]$, and let $\m\in\mmn$.
The {\it standard form} of $\m$ is a list of blocks $(i_1,j_1)/(i_2,j_2)/\ldots/(i_n,j_n)$ such that
$i_r<j_r$ for all $1\leq r\leq n$ and $1=i_1<i_2<\cdots <i_n$. Throughout this paper we always write $\m$ in standard form.
It is well known that $\m$ can
be regarded as a fixed-point-free involution on $[2n]$
and also as a Brauer diagram on $[2n]$ (see~\cite{Lovasz86} for instance).
Let $\so(\m)$ be the number of blocks of $\m$ with odd smaller entries.
It is well known that
\begin{equation}\label{Nnx-Matching}
N_n(x)=\sum_{\m\in\mmn}x^{\so(\m)},
\end{equation}
which has been studied in~\cite{Ma13,Ma1302}.
In particular, from~\cite[Theorem 9]{Ma13}, we have
\begin{equation*}\label{Anx-Nnx02}
2^nxA_n(x)=\sum_{k=0}^n\binom{n}{k}N_k(x)N_{n-k}(x).
\end{equation*}

Combining~\eqref{Nnx-Stirling},~\eqref{Nnx-exc-cyc} and~\eqref{Nnx-Matching}, we get
\begin{equation}\label{bijection04}
\sum_{\sigma\in\mqn}x^{\ap(\sigma)}=\sum_{\pi\in\msn}x^{n-\exc(\pi)}2^{n-\cyc(\pi)}=\sum_{\m\in\mmn}x^{\so(\m)}.
\end{equation}
It is natural
to consider the following question: Is there existing a unified combinatorial approach to prove~\eqref{bijection04}?
The main object of this paper is to provide a solution to this problem.
The paper is organized as follows. In Section~\ref{Section-2}, we present the main results, and recall
some definitions that will be used throughout the
rest of this work. In other Sections, we give combinatorial proofs of these main results.
\section{Definitions and main results}\label{Section-2}
Let $N_n(x)=\sum_{k=1}^nN({n,k})x^k$.
Hence $N(n,k)$ is the number of perfect matchings in $\mmn$ with the restriction that only $k$ matching pairs
have odd smaller entries.
We partition the blocks of $\m$ into
four subsets:
\begin{align*}
  \OO(\m)&=\{(a,b)\in\m\mid \textrm{a and b are both odd}\}, \\
  \OEE(\m)&=\{(a,b)\in\m\mid \textrm{a is odd and b is even}\}, \\
  \EO(\m)&=\{(a,b)\in\m\mid \textrm{a is even and b is odd}\}, \\
  \EE(\m)&=\{(a,b)\in\m \mid \textrm{a and b are both even}\}.
\end{align*}
Let $\oo(\m)=\#\OO(\m),\oee(\m)=\#\OEE(\m), \eo(\m)=\#\EO(\m), \ee(\m)=\#\EE(\m)$.
It is evident that $\so(\m)=\oo(\m)+\oee(\m)$. Here we present a further characterization of $\so(\m)$.
\begin{prop}\label{prop01}
We have $\so(\m)=\oee(\m)+\ee(\m)$, and the numbers $N(n,k)$ satisfy the recurrence relation
\begin{equation}\label{recurrence-01}
N({n+1,k})=2kN({n,k})+(2n-2k+3)N(n,k-1).
\end{equation}
for $n,k\geq 1$, where $N(1,1)=1$ and $N(1,k)=0$ for $k\geq 2$ or $k\leq 0$.
\end{prop}
\begin{proof}
Given a $\m\in\mmn$. Note that
$$2\oo(\m)+\#\oee(\m)+\#\eo(\m)=n=2\#\ee(\m)+\#\oee(\m)+\#\eo(\m).$$
Then $\#\oo(\m)=\#\ee(\m)$. Hence $\so(\m)=\oee(\m)+\ee(\m)$.
Therefore, we have
\begin{equation}\label{Nnk-Mnk}
N(n,k)=\#\{\m\in\mmn: \oee(\m)+\ee(\m)=k\}.
\end{equation}

We now prove~\eqref{recurrence-01}. Let $(a,b)$ be a given subset of $\m$. Let
$\varphi$ be the construction of $\m'\in\mmnn$ that replacing $(a,b)$ by $(a, 2n+1)/(b,2n+2)$ or $(a, 2n+2)/(b,2n+1)$. We distinguish two cases.
\begin{enumerate}
  \item [\rm ($c_1$)] If $(a,b)\in \OEE$ or $(a,b)\in \EE$, then the construction $\varphi$ does not increasing the number odd smallers.
  Combining~\eqref{Nnk-Mnk}, this accounts for $2kN(n,k)$ possibilities.
  \item [\rm ($c_2$)] If $(a,b)\in \OO$ or $(a,b)\in \EO$, then the construction $\varphi$ does form a new odd smaller. Moreover, we can also append $(2n+1,2n+2)$ to $\m$.
  This gives $(2n-2(k-1)+1)N(n,k-1)=(2n-2k+3)N(n,k-1)$ possibilities.
\end{enumerate}
\end{proof}

For $k\geq 1$ and $\ell\geq 0$, we define
$$P_k=\{1,2,3,\ldots, 2k\}~\textrm{and}~ N_{\ell}=\{-1,-2,-3,\ldots, -2\ell,\star\}.$$
Let $Y_n=(y_1,y_2,\ldots,y_n)$, where $y_i\in P_k\bigcup N_{\ell}$ for $1\leq i\leq n$.
In particular, $P_1\bigcup N_0=\{1,2,\star\}, P_1\bigcup N_1=\{1,2,-1,-2,\star\}$ and $P_2\bigcup N_0=\{1,2,3,4,\star\}$.
Let $\pos(Y_n)$ (resp. $\nega(Y_n)$) be the number of positive (resp. negative) entries of $Y_n$.
Denote by $\st(Y_n)$ the number of $\star$ of $Y_n$.
Careful consideration of Proposition~\ref{prop01} yields the following definition.
\begin{definition}\label{def01}
We call the sequence $Y_n$ a {\it $\my$-sequence} of length $n$ if $y_1=\star$ and
$y_k\in P_{1+s_k}\bigcup N_{t_k}$ for $2\leq k\leq n$, where $s_k=\nega(Y_{k-1})+\st(Y_{k-1})-1$ and
$t_k=\pos(Y_{k-1})$ for $k\geq 2$.
\end{definition}
Note that $s_2=t_2=0$. Hence $y_2\in P_1\bigcup N_0$.
For example, $(\star,1,-1,2,\star)$ is a $\my$-sequence, while $(\star,1,-1,-4,2)$ is not
since $y_4<-2$. Denote by $\myn$ the set of $\my$-sequences $Y_n$. Note that $1+s_k+t_k=k-1$ for $k\geq 2$.
Therefore, we have
$$\#\myn=(2(1+s_n)+(1+2t_n))\#\mynn=(2n-1)\#\mynn=(2n-1)!!.$$

We can now present the first main result of this paper.
\begin{theorem}\label{thm01}
For $n\geq 1$, we have
\begin{equation}\label{bijection01}
\sum_{\m\in\mmn}x^{\oee(\m)+\ee(\m)}=\sum_{Y_{n}\in \myn}x^{\nega(Y_{n})+\st(Y_{n})}=\sum_{\sigma\in\mqn}x^{\ap(\sigma)}.
\end{equation}
\end{theorem}

In this paper we always write $\pi\in\msn$
by its standard cycle decomposition, in which each
cycle is written with its smallest entry first and the cycles are written in ascending order of their
smallest entry.
\begin{definition}
Let $(c_1,c_2,\ldots,c_i)$ be one cycle of $\pi$. We say that $c_j$ is a {\it cycle descent} if $c_j>c_{j+1}$ for $1\leq j<i$.
\end{definition}
Let $\cdes(\pi)$ be the number of cycle descents of $\pi$.
For example, for $\pi=(1,3,\textbf{4},2)(5,7)(6)$, we have $\cdes(\pi)=1$.
For $\pi\in\msn$, it is clearly that $\exc(\pi)+\cyc(\pi)+\cdes(\pi)=n$. Therefore,
it follows from~\eqref{Nnx-exc-cyc} that
\begin{equation*}\label{Nnx-cdes}
N_n(x)=\sum_{\pi\in\msn}x^{\cyc(\pi)+\cdes(\pi)}2^{n-\cyc(\pi)}.
\end{equation*}

Now we present the second main result of this paper.
\begin{theorem}\label{thm05}
For $n\geq 1$, we have
\begin{equation}\label{bijection03}
\sum_{Y_{n}\in \myn}x^{\nega(Y_{n})}y^{\st(Y_{n})}z^{\pos(Y_n)}=\sum_{\pi\in\msn}(2x)^{\cdes(\pi)}y^{\cyc(\pi)}(2z)^{\exc(\pi)}.
\end{equation}
\end{theorem}

Note that $\nega(Y_{n})+\st(Y_{n})+\pos(Y_n)=n$.
The following corollary is immediate.
\begin{cor}\label{cor-Anx}
For $n\geq 1$, we have
$$A_n(x)=\sum_{Y_{n}\in \myn}\left(\frac{1}{2}\right)^{\nega(Y_{n})+\pos(Y_n)}x^{\pos(Y_n)}.$$
Equivalently,
$$2^nA_n(x)=\sum_{Y_{n}\in \myn}2^{\st(Y_{n})}x^{\pos(Y_n)}.$$
\end{cor}

In the following sections, we give bijective proofs of the main results. As a consequence, we get a desired proof of~\eqref{bijection04}.
\section{Proof of the left equality of~\eqref{bijection01}}\label{Section-3}
In the following discussion, we always put any 2-elements block of $\m$ into exactly a parenthesis or a square bracket.
For $(a,b)\in\m$, if $(a,b)\in \OEE(\m)$ or $(a,b)\in \EE(\m)$, then we replace $(a,b)$ by $[a,b]$.
Otherwise, the parentheses that contain $a$ and $b$ unchanged.
For example, we replace $(1,3)/(2,4)/(5,8)/(6,7)$ by $(1,3)/[2,4]/[5,8]/(6,7)$.
Set $m=2(1+s_k+t_k)$. Let $\mm\left(P_{1+s_k}\bigcup N_{t_k}\right)$ be the set of perfect matchings of $[m]$ with exactly $1+s_k$ square brackets and $t_k$ parentheses.

Now we present a bracket-breaking algorithm (\BBA ~for short):
\begin{itemize}
  \item [\rm ($S_1$)] Note that $\mm(P_1\bigcup N_0)=[1,2]$. We take $[1,2]$ as the starting point, which corresponds to the first term of a $\my$-sequence. Since $P_1\bigcup N_0=\{1,2,\star\}$, we can perform one of the following three operations:
  \begin{itemize}
    \item [\rm ($c_1$)] Using the element 1 of $P_1\bigcup N_0$ to break the square bracket $[1,2]$ such that $[1,2]$ is replaced by $(1,3)/[2,4]$.
    \item [\rm ($c_2$)] Using the element 2 of $P_1\bigcup N_0$ to break the square bracket $[1,2]$ such that $[1,2]$ is replaced by $[1,4]/(2,3)$.
    \item [\rm ($c_3$)] As for $\star$, we append $[3,4]$ right after $[1,2]$.
  \end{itemize}
  \item [\rm ($S_2$)] For $k\geq 2$, given a $\m\in \mm(P_{1+s_k}\bigcup N_{t_k})$.
We use entries of $P_{1+s_k}\bigcup N_{t_k}$ to break some square brackets or parentheses of $\m$.
Assume that $c,d,e$ and $f$ are positive integers.
For the entries of $P_{1+s_k}$, we distinguish two cases:
\begin{itemize}
  \item [\rm ($c_1$)] Using the element $2i-1$ to break the $i$-th square bracket with the restriction that (i) if the $i$-th square bracket with elements $2c-1$ and $2d$, then we replace $[2c-1,2d]$ by $(2c-1,m+1)/[2d, m+2]$; (ii) if the $i$-th square bracket with elements $2e$ and $2f$, then we replace $[2e,2f]$ by $(2e,m+1)/[2f,m+2]$.
  \item [\rm ($c_2$)] Using the element $2i$ to break the $i$-th square bracket with the restriction that (i) if the $i$-th square bracket with elements $2c-1$ and $2d$, then we replace $[2c-1,2d]$ by $[2c-1,m+2]/(2d, m+1)$; (ii) if the $i$-th bracket with elements $2e$ and $2f$, then we replace $[2e,2f]$ by $[2e,m+2]/(2f,m+1)$.
\end{itemize}
For the entries of $N_{t_k}$, we distinguish three cases:
\begin{itemize}
  \item [\rm ($c_1$)] For $1\leq i\leq t_k$, we use the element $-2i+1$ to break the $i$-th parenthesis with the restriction that (i) if the $i$-th parenthesis with elements $2c-1$ and $2d-1$, then we replace $(2c-1,2d-1)$ by $(2c-1,m+1)/[2d-1, m+2]$;
      (ii) if the $i$-th parenthesis with elements $2e$ and $2f-1$, then we replace $(2e,2f-1)$ by $(2e,m+1)/[2f-1,m+2]$.
  \item [\rm ($c_2$)] For $1\leq i\leq t_k$, we use the element $-2i$ to break the $i$-th parenthesis with the restriction that
  (i) if the $i$-th parenthesis with elements $2c-1$ and $2d-1$, then we replace $(2c-1,2d-1)$ by $[2c-1,m+2]/(2d-1, m+1)$;
  (ii) if the $i$-th parenthesis with elements $2e$ and $2f-1$, then we replace $(2e,2f-1)$ by $[2e,m+2]/[2f-1,m+1]$.
   \item [\rm ($c_3$)] As for $\star$, we append $[m+1,m+2]$ right after $\m$.
\end{itemize}
\end{itemize}

Given a $\my$-sequence $Y_n$ with $\nega(Y_{n})+\st(Y_{n})=i$.
Repeat the \BBA ~$n$-times, we can get a unique perfect matching $\m$ of $[2n]$ with
$\oee(\m)+\ee(\m)=i$.
Conversely, given a perfect matching $\m\in\mmn$ with
$\oee(\m)+\ee(\m)=i$. If we delete $2n-1$ and $2n$,
then we can find the $n$-th element of the corresponding $\my$-sequence. Along the same lines, we can get a unique $\my$-sequence $Y_n$ with $\nega(Y_{n})+\st(Y_{n})=i$.
Thus the \BBA ~gives a bijective proof of the left equality of~\eqref{bijection01}. In fact, using the \BBA, we give a bijective proof of the following result:
\begin{equation*}
\sum_{\m\in\mmn}x^{\oee(\m)+\ee(\m)}y^{\eo(\m)+\oo(\m)}=\sum_{Y_{n}\in \myn}x^{\nega(Y_{n})+\st(Y_{n})}y^{\pos(Y_n)}.
\end{equation*}

\begin{ex}
Let $Y_5=(\star,1,-2,4,-1)$ and let $\m=[1,6]/[2,8]/(3,9)/(4,7)/[5,10]$.
The correspondence between $Y_5$ and $\m$ is built up as follows:
\begin{align*}
1&\rightarrow [1,2]\Leftrightarrow (1,3)/[2,4];\\
-2&\rightarrow (1,3)/[2,4]\Leftrightarrow [1,6]/(3,5)/[2,4]=[1,6]/[2,4]/(3,5);\\
4&\rightarrow [1,6]/[2,4]/(3,5)\Leftrightarrow [1,6]/[2,8]/(4,7)/(3,5)=[1,6]/[2,8]/(3,5)/(4,7);\\
-1&\rightarrow [1,6]/[2,8]/(3,5)/(4,7)\Leftrightarrow [1,6]/[2,8]/(3,9)/[5,10]/(4,7)=[1,6]/[2,8]/(3,9)/(4,7)/[5,10].\\
\end{align*}
\end{ex}
\section{Proof of the right equality of~\eqref{bijection01}}\label{Section-4}
Set $r=1+s_k+t_k$. Denote by $\mq\left(P_{1+s_k}\bigcup N_{t_k}\right)$ the set of Stirling permutations of order $r$ with exactly $1+s_k$ ascent plateaus.
In particular, $\mq \left(P_1\bigcup N_0\right)=\{11 \}$, $\mq \left(P_1\bigcup N_1\right)=\{2211,1221\}$ and $\mq \left(P_2\bigcup N_0\right)=\{1122\}$.
Let us first give a definition of labeled Stirling permutations.
\begin{definition}
Let $\sigma\in \mq\left(P_{1+s_k}\bigcup N_{t_k}\right)$.
If $i_1<i_2<\ldots<i_{1+s_k}$ are the ascent plateaus of $\sigma$, then we put the superscript labels $2\ell-1$ before $i_{\ell}$ and $2\ell$ after it,
where $1\leq \ell\leq 1+s_k$. In the remaining positions, we put the superscript labels $-1,-2,\ldots,-2t_k$ and $\star$ from left to right.
\end{definition}
For example, the labels of 13324421 are given as follows: $$^{-1}1^{1}3^{2}3^{-2}2^{3}4^{4}4^{-3}2^{-4}1^{\star}.$$

Given a $\my$-sequence $Y_n=(y_1,y_2,\ldots,y_n)$.
Now we present a labeled Stirling permutations algorithm (\LSPA ~for short):
\begin{enumerate}
  \item [\rm ($S_1$)] Since $\mq\left(P_{1}\bigcup N_{0}\right)=\{11\}$,
  we take $11$ as the start point, which corresponds to $y_1=\star$. Since $11$ can be labeled as $^{1}1^{2}1^{\star}$, we distinguish three cases:
if $y_2=1$, then we get 2211 by inserting 22 to the position with superscript label 1;
if $y_2=2$, then we get 1221 by inserting 22 to the position with superscript label 2;
if $y_2=\star$, then we we get 1122 by inserting 22 to the position with superscript label $\star$.
  \item [\rm ($S_2$)] When $y_2=1$, we take $^{1}2^{2}2^{-1}1^{-2}1^{\star}$ as the start point.
  Since $y_3\in P_1\bigcup N_1$, we distinguish five cases: if $y_3=1$, then we get $\textbf{3}32211$ by inserting 33
  to the position with superscript label 1; if $y_3=2$, then we get $2\textbf{3}3211$ by inserting 33 to the position
  with superscript label 2; if $y_3=-1$, then we get $\textbf{2}2\textbf{3}311$ by inserting 33 to the position with superscript label -1;
if $y_3=-2$, then we get $\textbf{2}21\textbf{3}31$ by inserting 33 to the position with superscript label -2;
  if $y_3=\star$, then we get $\textbf{2}211\textbf{3}3$ by inserting 33 to the position with superscript label $\star$.
  \item [\rm ($S_3$)] When $y_2=2$, we take $^{-1}1^{1}2^{2}2^{-2}1^{\star}$ as the start point. Note that $y_3\in P_1\bigcup N_1$.
  Along the same lines as the five cases of ($S_2$), we get $133221,123321, $331221$, 122331,122133$.
  \item [\rm ($S_4$)] When $y_2=\star$, we take $^{1}1^{2}1^{3}2^{4}2^{\star}$ as the start point. Note that $y_3\in P_2\bigcup N_1$.
  Similarly, we get 331122,133122,113322112332,112233.
 \item [\rm ($S_5$)]
  Repeat the above procedure, we get all labeled Stirling permutations.
  \end{enumerate}

It is straightforward
to show that each such labeled Stirling permutation will be obtained exactly once. Indeed, given a labeled Stirling permutation of order $n$ with $i$ ascent plateaus,
we can just read the indices of ascent plateaus. Deleting the pair $(2n)(2n)$ gives the entry $y_n$. Along the same lines,
we can get a unique $\my$-sequence $Y_n$ with $\nega(Y_{n})+\st(Y_{n})=i$.
In conclusion, \LSPA~ gives the desired proof of the right equality of~\eqref{bijection01}.

\begin{ex}
Let $Y_6=(\star,2,-1,3,-2,2)$ and let $\sigma=366314455221$.
The correspondence between $Y_6$ and $\sigma$ is built up as follows:
\begin{align*}
2&\rightarrow ^{1}1^{\textbf{2}}1^{\star}\Leftrightarrow 1221;\\
-1&\rightarrow  ^{\textbf{-1}}1^{1}2^{2}2^{-2}1^{\star}\Leftrightarrow 331221;\\
3&\rightarrow ^{1}3^{2}3^{-1}1^{\textbf{3}}2^{4}2^{-2}1^{\star}\Leftrightarrow 33144221;\\
-2&\rightarrow ^{1}3^{2}3^{-1}1^{3}4^{4}4^{\textbf{-2}}2^{-3}2^{-4}1^{\star}\Leftrightarrow 3314455221;\\
2&\rightarrow ^{1}3^{\textbf{2}}3^{-1}1^{3}4^{4}4^{5}5^{6}5^{-2}2^{-3}2^{-4}1^{\star}\Leftrightarrow 366314455221.
\end{align*}
\end{ex}

\section{Proof of~\eqref{bijection03}}\label{Section-5}
\quad

Let us first give a definition of labeled permutations.
\begin{definition}
Let $\pi\in\msn$ with $p$ excedances.
If $i_1<i_2<\cdots <i_p$ are the excedances,
then we put superscript labels $-k$ between $i_k$ and $\pi(i_k)$, where $1 \leq k\leq p$.
In the remaining positions except the first position, we put the superscript labels $1,2,\ldots,n-p$ from left to right.
\end{definition}
In the following discussion, we shall attach superscript labels to all permutations in $\msn$.
For example, the labels of $(\textbf{1},\textbf{3},7,4,2)(5)(\textbf{6},10,8)(\textbf{9},11)$ are given as follows:
$$(\textbf{1}~^{-1}~\textbf{3}^{-2}~7~^1~4~^2~2~^3)(5^4)(\textbf{6}~^{-3}~10~^5~8~^6)(\textbf{9} ~^{-4}~11~^7).$$

A labeled permutations algorithm (\LPA ~for short) for generating $\my$-sequences is given as follows:

Take $(1^1)$ as the start point, which corresponds to the first term $\star$ of a $\my$-sequence.
Let $\pi$ be a labeled permutation in $\msn$. There are $n+1$ permutations of $\msnn$ can be obtained from $\pi$
by inserting $n+1$ to the positions with
superscript labels or as a new cycle $(n+1)$. We distinguish three cases:
\begin{enumerate}
    \item [\rm ($C_1$)]If we insert the entry $n+1$ to the position with superscript label $k$, then the $n$-th term of the corresponding $\my$-sequence is $y_n\in\{2k-1,2k\}$.
    \item [\rm ($C_2$)]If we insert the entry $n+1$ to the position with superscript label $-\ell$, then the $n$-th term of the corresponding $\my$-sequence is $y_n\in\{-2\ell+1,-2\ell\}$.
    \item [\rm ($C_3$)]If we insert the entry $n+1$ at the end of $\pi$ to form a new cycle $(n+1)$, then the $n$-th term of the corresponding $\my$-sequence is $\star$.
  \end{enumerate}
In particular, consider $\pi=(1^1)$. Note that $(1,2)$ is obtained from $(1^1)$ by inserting 2 to the position with superscript label 1.
Hence $y_2\in \{1,2\}$ is the $\my$-sequence that corresponds to $(1,2)$.
Note that $(1)(2)$ is obtained from $(1^1)$ by inserting 2 as a new cycle $(2)$. Hence
$y_2=\star$ is the $\my$-sequence that corresponds to $(1)(2)$.
Take $(1^1)$ as a start point. Repeat the \LPA~ $n$ times, it is easy to verify
that each $\my$-sequence of length $n$ will be obtained exactly once.
Note that $2^{\cdes(\pi)+\exc{\pi}}$ is the number of $\my$-sequences that corresponds to $\pi$. Therefore, the \LPA~ gives a combinatorial proof of~\eqref{bijection03}.

\begin{ex}
Given $\pi=(1,3,5,2,6)(4)$. The \LPA~ can be done if you proceed as follows:
\begin{align*}
(1^1)&\rightarrow (1,2)\Leftrightarrow y_2\in\{1,2\};\\
(1~^{-1}~2~^1)&\rightarrow  (1,3,2)\Leftrightarrow y_3\in\{-1,-2\};\\
(1~^{-1}~3~^1~2~^2)&\rightarrow (1,3,2)(4)\Leftrightarrow y_4\in\{\star\};\\
(1~^{-1}~3~^{1}~2~^{2})(4~^{3})&\rightarrow (1,3,5,2)(4)\Leftrightarrow y_5\in\{1,2\}; \\
(1~^{-1}~3~^{-2}~5~^{1}~2^{2})(4^{3})&\rightarrow (1,3,5,2,6)(4) \Leftrightarrow y_6\in\{3,4\}.
\end{align*}
Therefore, the corresponding $\my$-sequences of $\pi$ are given as follows:
$$(\star,1,-1,\star,1,3);(\star,1,-1,\star,1,4);(\star,1,-1,\star,2,3);(\star,1,-1,\star,2,4);$$
$$(\star,1,-2,\star,1,3);(\star,1,-2,\star,1,4);(\star,1,-2,\star,2,3);(\star,1,-2,\star,2,4);$$
$$(\star,2,-1,\star,1,3);(\star,2,-1,\star,1,4);(\star,2,-1,\star,2,3);(\star,2,-1,\star,2,4);$$
$$(\star,2,-2,\star,1,3);(\star,2,-2,\star,1,4);(\star,2,-2,\star,2,3);(\star,2,-2,\star,2,4).$$
\end{ex}

\end{document}